\documentclass{article}
\usepackage[utf8]{inputenc}
\usepackage{amsmath}
\usepackage{amsthm}
\usepackage{mathrsfs}
\usepackage{hyperref}
\usepackage{graphicx,float}
\usepackage[english]{babel}
\usepackage{setspace}
\usepackage{blindtext}
\usepackage{authblk}
\newtheorem{defn}{Definition}
\newtheorem{theorem}{Theorem}
\newtheorem{lemma}[theorem]{Lemma}
\newtheorem{corollary}{Corollary}[theorem]
\title{Eigenvalue Interlacing of Bipartite Graphs and Construction of Expander Code using Vertex-split of a Bipartite Graph}
\author{ Machasri Manickam\\ email \href{mailto:machasri.m2019@vit.ac.in}{machasri.m2019@vit.ac.in} 
   \and Kalyani Desikan \\ email \href{mailto:kalyanidesikan@vit.ac.in}{kalyanidesikan@vit.ac.in} }
\affil{Vellore Institute of Technology, Chennai, India.}

\begin{document}

\maketitle
\begin{abstract}
The second largest eigenvalue of a graph is an important algebraic parameter which is related with the expansion, connectivity and randomness properties of a graph. Expanders are highly connected sparse graphs. In coding theory, Expander codes are Error Correcting codes made up of bipartite expander graphs. In this paper, first we prove the interlacing of the eigenvalues of the adjacency matrix of the bipartite graph with the eigenvalues of the bipartite quotient matrices of the corresponding graph matrices. Then we obtain bounds for the second largest and second smallest eigenvalues. Since the graph is bipartite, the results for Laplacian will also hold for Signless Laplacian matrix. We then introduce a new method called vertex-split of a bipartite graph to construct asymptotically good expander codes with expansion factor $\frac{D}{2}<\alpha < D$ and $\epsilon<\frac{1}{2}$ and prove a condition for the vertex-split of a bipartite graph to be $k-$connected with respect to $\lambda_{2}.$ Further, we prove that the vertex-split of $G$ is a bipartite expander. Finally, we construct an asymptotically good expander code whose factor graph is a graph obtained by the vertex-split of a bipartite graph.
\end{abstract}
\textbf{Keywords:}   expander code, vertex-split, second largest eigenvalue, bipartite graph, quotient matrix.\\
\textbf{MSC Classification :} 05C40, 05C48, 05C50.
\section{Introduction}
Let $G$ be a finite graph. The adjacency matrix $A(G)$ of $G$ is an $n \times n$ matrix $A=[a_{ij}]$, where $a_{ij}=	1$ if $v_{i}$ and $v_{j}$ are adjacent, otherwise it is $0$. Let $\lambda_{1}\geq \lambda_{2} \geq \dots \geq \lambda_{n}$ be the eigenvalues of $A$ known as the spectrum of $G.$ The Laplacian matrix of $G$ is $L(G) = D(G) - A(G)$ where $D(G)$ is the diagonal degree matrix. Let $\mu_{1}\geq \mu_{2}\geq \mu_{3}\geq \dots \geq \mu_{n-1}\geq \mu_{n}$ be the eigenvalues of the Laplacian matrix. The signless Laplacian matrix of $G$ is $Q (G) = D(G) + A(G)$. For a bipartite graph, Laplacian and signless Laplacian eigenvalues are the same.\\
 \\ 
\indent Expanders are graphs which are sparse but highly connected. The word sparse means that the number of edges of $G$ is much less than the possible number of edges of $G$. Expander has wide application in various areas of computer science including  pseudorandomness, complexity theory, coding theory, algorithm design, cryptography, etc.\\

Depending on the use, the term expander has many meanings. An 
$(n,m,d,\gamma,\alpha)-$ \textit{expander} is a bipartite graph $H=(X,Y,E)$ where $|X|=n,|Y|=m,$ $d(x)=d$ for all $x\in X$ and for every $S\subseteq X$, $|S|\leq \gamma n$, we have the set of vertices $N(S)\subseteq Y$ such that $|N(S)|\geq \alpha |S|,$ where $\gamma$ and $\alpha$ are positive constants. To obtain good expansion, $\alpha$ should be high. An expander is said to be a lossless expander when $\alpha$ is closer to $d.$ Expander code has rate at least $1-\frac{m}{n}.$ Therefore, smaller $m$ implies codes with higher rate.\\

An Error-Correcting Code (ECC) is a type of encoding used in the theory of coding to transfer messages as binary numbers in a way that allows the message to be decoded even if some bits are reversed. The error-correcting codes known as expander codes are constructed from bipartite expander graphs.\\

A collection of strings known as codewords form an error correcting code denoted by $C$. Block length of an ECC is the number of elements in the code word and it is denoted by $n$. The codewords consist of $n$ symbols from $\sigma$ which is the alphabet set. A code is referred to as $(n,k)_{q}$ code where $|\sigma | =q$ and $|C|=q^{k}.$ $k$ indicates how many informational symbols are contained in each codeword and $R(C)=\frac{k}{n}$ is the rate of the code. The smallest Hamming distance between two different codewords of $C$ is the distance $D(C)$ of the code. The parity-check and generator matrix views of the code give reasons for seeing linear codes as graphs and building them using graph-theoretic methods.
To express the properties of the code from the characteristics of the graph, we can interpret the parity check matrix for a $[n,n-m]_{2}$ code as representing an $n\times m$ bipartite graph. Low Density Parity Check (LDPC) codes are an interesting family of codes since they appear as sparse graphs in the graph view because there are few $1's$ in each row and column of the parity-check matrix.\\

In Section \ref{sec3} we obtain the upper bound for the second largest adjacency eigenavalue, lower bound for the second smallest adjacency eigenvalue and upper bound for the second largest Laplacian eigenvalue of a connected bipartite graph $G.$ To obtain sharper bounds, we discuss some possible cases with respect to the bipartitions of $G$ in a minimally connected bipartite graph and derive bounds for the second largest adjacency eigenvalue and second smallest adjacency eigenvalue in terms of $n$, the number of vertices.\\

In Section \ref{sec4}, we introduce a new concept called vertex-split of bipartite graph and we prove a condition for the vertex-split of a bipartite graph to be $k-$connected with respect to $\lambda_{2}.$ We prove theorems related to connectivity and prove the expansion of vertex-split of a bipartite graph. Finally, in Section \ref{sec5} we show that the vertex-split of a biregular bipartite graph forms an expander code.

\section{Related work}
Eigenvalues are often difficult to compute. Therefore, obtaining bounds for eigenvalues is useful. In literature, bounds have been found for the second largest eigenvalue of some family of graphs. Chang An \cite{CH} obtained upper and  lower bounds for the second largest eigenvalue of a tree. The lower bounds on the second largest eigenvalue of a regular graph with given girth were obtained by Patrik Solv \cite{PS}. Mehatari and Kannan \cite{RR} derived bounds for the second largest and second smallest eigenvalues of adjacency matrices, normalized adjacency matrices and Laplacian matrices of regular graphs. In 1988, Powers \cite{PW} gave some upper bounds of second largest eigenvalue for general graphs and bipartite graphs. In 2012, Mingqing Zhaim et al. \cite{MHB} presented upper bounds for the second largest eigenvalue of connected graphs and particularly, for bipartite graphs.\\

Expander codes which are constructed from unbalanced bipartite expander graphs with expansion factor $(1-\epsilon)D$ for $\epsilon< \frac{1}{2}$ are said to be asymptotically good codes. Initially, constructions with $\alpha \approx \frac{D}{2}$ were known explicitly. So the requirement is to to give explicit construction of codes with expansion higher than $\frac{D}{2}$. In 2002, Capalbo et al. \cite{Cap} presented an explicit construction with expansion $(1-\epsilon)D$ where $D$ is the degree of the left side partition of $G$ for any desired $\epsilon> 0$, and imbalance ratio $\frac{m}{n}$. Using the edge vertex  matrix is one method of building an unbalanced expander, which Tanner \cite{tan} introduced and Sipser and Spielman \cite{ss}-\cite{daniel} employed. Also, it is observed by Zemor \cite{zemor} that if the edge vertex incidence graph's underlying graph is a bipartite graph, the decoding technique is straightforward.\\ 
In this work, we use vertex-split of a bipatite graph to construct expanders codes with expansion factor $\frac{D}{2}<\alpha < D$ and $\epsilon<\frac{1}{2}$.
\section{Preliminaries}
\begin{defn} \cite{CHL}(Quotient matrix)
Let $$M=\begin{bmatrix} 
M_{11}& \cdots & M_{1t}\\
\vdots & \ddots & \vdots\\
M_{t1}& \cdots& M_{tt}
\end{bmatrix}
$$be a real matrix of order n and $M_{ij}$ be the blocks of M, where i,j $= {1,2,\dots ,t}$. Then B(M)=($b_{ij}$) is called the Quotient Matrix of M where $b_{ij}$ is the sum of
all entries in $M_{ij}$ divided by the number of rows of $M_{ij}$.
\end{defn}
\begin{defn} \cite{BH}(Interlacing)
Consider two sequences of real numbers: \\ $\xi_{1},\xi_{2},\dots, \xi_{n}$ and $\eta_{1},\eta_{2},\dots,\eta_{m}$ with $m \leq n$.
The second sequence is said to interlace the first one whenever $\xi_{i} \leq \eta_{i} \leq \xi_{n-m+i}$ for $i = 1, 2,\dots,m.$ 
The interlacing is called tight if there exists an integer $k \in [0,m]$ such that $\xi_{i} = \eta_{i}$ for $1 \leq i \leq k$ and $\xi_{n-m+i} = \eta_{i}$ for $k + 1\leq i \leq m$.
\end{defn}
\begin{defn}
A graph is said to be minimally connected if removal of any one edge disconnects the graph.
\end{defn}
\begin{lemma}\label{L1}{(\cite{WH1}-\cite{HMF})}
Let ${A_{Q}}$ be the quotient matrix of a symmetric matrix $A$ whose rows and columns are partitioned according to a partitioning 
\\ $(X_{1} , X_{2} , \dots , X_{m} ).$ Then 
\begin{center}
\begin{enumerate}
\begin{enumerate}
\begin{enumerate}
   \item The eigenvalues of ${A_{Q}}$ interlace the eigenvalues of $A$.
   \item If the interlacing is tight then the partition is equitable.
\end{enumerate}
\end{enumerate}
\end{enumerate}
\end{center}
\end{lemma}
\begin{defn}\label{ver}\cite{HLW} (Vertex expander)
A graph $G$ with $n$ vertices is said to be a vertex expander if $|N(S)|\geq A |S|$ for all $S\subseteq V : |S| \leq \frac{n}{2}$, where $A$ is a constant and $N(S)$ is the neighbourhood of $S$ in $G$ not in ${S}$.
\end{defn}
\begin{defn}\label{spec}\cite{HLW} (Spectral expansion) 
 The spectral expansion of graph $G$ is defined by $\lambda=max\{|\lambda_{2}|, |\lambda_{n}|\}$.
\end{defn}
\begin{defn}\label{bip}\cite{AL}
Let $G=(V,E)$ be a graph. For a subset $S$ of $V$ let
$N(S) = \{v\in V: vs\in E$ for some $s\in S\}$.
An $(n,d,c)-$expander is a bipartite graph on the sets of vertices $X$ and $Y$,
where $|X|=|Y|=n$ , the maximal degree of a vertex is $d$, and for every set
$S\subseteq X$ of cardinality $|S|=\alpha\leq \frac{n}{2}$, $|N(S)|\geq (1+c(1-\frac{\alpha}{n}))\alpha.$
\end{defn}

\begin{defn} \cite{Cap}
A $D-$left regular bipartite graph $G=(X\cup Y, E)$ where $|X|=n$ and $|Y|=m$ such that for all $S\subseteq X $ with $|S|\leq \gamma n,$ $|N(S)|\geq \alpha |S|,$ where $\gamma$ and $\alpha$ are positive constants is known as an $(n,m,D,\gamma, D(1-\epsilon))$ expander.
\end{defn}
\noindent \textbf{Note :}\\
Let $\lambda_{1}\geq \lambda_{2}\geq \dots \geq \lambda_{n}$ be the adjacency eigenvalues of $G$ and let $\lambda_{1}'\geq \lambda_{2}'\geq \dots \geq \lambda_{n}'$ be the adjacency eigenvalues of $G'$. Assume that $\gamma_{i}=\frac{\lambda_{i}'}{d}$ be the normalized eigenvalues of $G'$.
Let $\lambda'$ be the spectral expansion of $G'.$ Then $\gamma=\frac{\lambda'}{d}$ is the spectral expansion of $G'$ with respect to the normalized adjacency eigenvalues.
\begin{lemma}\cite{vertosp}\label{T2} (vertex expansion to spectral expansion). Let G be a $d-$regular graph. For
every $\epsilon>0$ and $d > 0$, there exists $\gamma > 0$ such that if $G$ is a $d-$regular $(1 + \epsilon)-$expander then $G$ has spectral expansion $(1-\gamma).$ Specifically, we can take $\gamma = 
\Omega(\epsilon^{2}/d)$.
\end{lemma}
\begin{lemma}\cite{guru}\label{dist}
Let $G$ be an $(n,m,D,\gamma, D(1-\epsilon))$ expander. Then the distance of the code corresponding to the graph $G$ is $\Delta(C(G))\geq 2\gamma(1-\epsilon)n$ where $C(G)$ denotes the code corresponding to an expander graph $G.$
\end{lemma}
\section{Eigenvalue Interlacing of Bipartite Quotient Matrix and Spectral Bounds}\label{sec3}
In this section we derive the upper bounds for the second largest eigenvalues of both the adjacency matrix and Laplacian matrix of a connected bipartite graph. Also, we derive the lower bounds for the second smallest eigenvalue of the adjacency matrix of a connected bipartite graph. To arrive at these bounds, we define the Bipartite quotient matrix as follows :
\begin{defn}(Bipartite quotient matrix)
Let $G$ be a bipartite graph. The Bipartite quotient matrix of the Adjacency matrix $A$ (Laplacian matrix $L$)  of $G$ is the Quotient matrix of $A/L$ whose rows and columns are partitioned according to the bipartition of $G$.
\end{defn}
\subsection{Bounds for the second largest eigenvalue and second smallest eigenvalue of adjacency matrix}\label{sec31}
In this section we derive the upper bounds for the second largest eigenvalue and lower bounds for the second smallest eigenvalue of the adjacency matrix of a connected bipartite graph $G$.
\begin{theorem}\label{T1}
Consider a Bipartite graph $G$ with bipartition $V=(X,Y)$ where $|X|=n_{1}$, $|Y|=n_{2}$ and $|V|=n_{1}+n_{2}=n$. Let $\lambda_{1}\geq\lambda_{2}\geq \dots \geq \lambda_{n-1}\geq \lambda_{n}$ be the eigenvalues of the adjacency matrix $A$ of $G$, $A_{Q}$ be the bipartite quotient matrix of $A$ and $\eta_{1}$ and $\eta_{2}$  the eigenvalues of $A_{Q}$. Then 
\begin{center}
\begin{enumerate}
\begin{enumerate}
\begin{enumerate}
    \item $\lambda_{1}\geq \eta_{1}\geq \lambda_{2}$
    \item $\lambda_{n-1}\geq \eta_{2}\geq \lambda_{n}$
    \item $\lambda_{2} \leq \frac{m}{\sqrt{n_{1}n_{2}}}$\label{subd3}
    \item $\lambda_{n-1} \geq \frac{-m}{\sqrt{n_{1}n_{2}}}$
\end{enumerate}
\end{enumerate}
\end{enumerate}
\end{center}
\end{theorem}
\begin{proof}
 Let $A$ be the adjacency matrix of $G$ represented in the following block matrix form with respect to the bipartition $V=(X,Y)$.
\begin{center}
$A=\begin{bmatrix}
A_{11} & A_{12} \\
A_{21} & A_{22}
\end{bmatrix} .$
\end{center}
Let $A_{Q}$ be the bipartite quotient matrix of $A.$ Then 
\begin{center}
 $A_{Q}=
\begin{bmatrix}
0 & \frac{m}{n_{1}} \\
\frac{m}{n_{2}} & 0
\end{bmatrix}.$   
\end{center}
The characteristic equation of $A_{Q}$ is 
$\lambda^{2}-\frac{m^{2}}{n_{1}n_{2}}=0$
and the corresponding eigenvalues of $A_{Q}$ are $\eta_{1}=\frac{m}{\sqrt{n_{1}n_{2}}}$ and $\eta_{2}=\frac{-m}{\sqrt{n_{1}n_{2}}}$.
To get the generalized interlacing, we need $n$ eigenvalues. To obtain this, consider the following.\\
The characteristic matrix of $G$ is given by 
$\tilde S= \begin{bmatrix}
J_{n_{1}\times 1}& 0_{n_{1}\times 1}\\
0_{n_{2}\times 1}& J_{n_{2}\times 1}.
\end{bmatrix}_{n\times 2}$ \\ where $J$ is the all-ones matrix.
Let $S=\tilde{S}K^{\frac{-1}{2}}$ where $K=diag(|X|,|Y|)$. i.e., $K=\begin{bmatrix}
n_{1} & 0 \\
0 & n_{2}
\end{bmatrix}.$ \\ 
Then $S= \begin{bmatrix}
P_{n_{1}\times 1}& 0_{n_{1}\times 1}\\
0_{n_{2}\times 1}& R_{n_{2}\times 1}
\end{bmatrix}_{n\times 2}.$\\
where $P_{n_{1}\times 1}$ and $R_{n_{2}\times 1}$ are column matrices with entries $(p_{i1})=\frac{1}{\sqrt{n_{1}}}$ for all $v_{i}\in X$ and  $(r_{i1})=\frac{1}{\sqrt{n_{2}}}$ for all $v_{i}\in Y$, respectively.\\
From the proof of Lemma \ref{L1} \cite{WH2} we have,
\begin{equation}\label{EQ1}
A_{Q}=S^{T}AS.    
\end{equation}
Left and right multiplying equation (\ref{EQ1}) by $S$ and $S^{T}$ respectively, we get, $SA_{Q}S^{T}=A$.\\
Now consider $SA_{Q}S^{T}$ and denote it by $C$.\\
Then $C=(c_{ij})=\begin{cases}
\frac{m}{n_{1}\sqrt{n_{1}}} & \text{$v_{i} \in X$ and $v_{j} \in Y$}\medskip \\
\frac{m}{n_{2}\sqrt{n_{2}}} & \text{$v_{i} \in Y$ and $v_{j} \in X$}\medskip \\ 
0 &\text{otherwise}.
\end{cases}$\\ 
The above matrix can be represented as a block matrix as follows:
\begin{center}
$C= \begin{bmatrix}
0 & M \\
N & 0 
\end{bmatrix}_{n\times n}.$
\end{center}
Denote the eigenvalues of $C$ by $\gamma_{i}$, $i=1,2,\ldots,n$. The eigenvalues of $C$ are the square roots of the non-zero eigenvalues of $MN$. That is $\gamma_{1}=\frac{m}{\sqrt{n_{1}n_{2}}}$, $\gamma_{n}=\frac{-m}{\sqrt{n_{1}n_{2}}}$ and $\gamma_{i}=0$ for $i=2,3,\dots, n-1$. Then by interlacing we get,\\
$\lambda_{1}\geq \gamma_{1}\geq \lambda_{2}\geq \dots \geq \lambda_{n-1}\geq \gamma_{n-1}\geq \lambda_{n}\geq \gamma_{n}$.\\ \\
By comparing the eigenvalues of $A_{Q}$ and $C$ we get, \\
$\gamma_{1}=\frac{m}{\sqrt{n_{1}n_{2}}}=\eta_{1}$ and $\gamma_{n}=\frac{-m}{\sqrt{n_{1}n_{2}}}=\eta_{2}$.\\ \\
Now let us first consider $\lambda_{1}\geq \gamma_{1}\geq \lambda_{2}$.
Since $\gamma_{1}=\eta_{1}$ we have $\lambda_{1}\geq \eta_{1}\geq \lambda_{2}$.\\
This proves $(\romannumeral 1)$.\\  \\
Since $G$ is bipartite, its eigenvalues are symmetric about the origin. Now from $(\romannumeral 1)$ and since $G$ is bipartite, we get  $\lambda_{n}\leq \eta_{2}\leq \lambda_{n-1}$. This proves $(\romannumeral 2)$.\\ \\
To prove $(\romannumeral 3)$, using $(\romannumeral 1)$ we have, 
\begin{center}
    $\lambda_{2}\leq \eta_{1}=\frac{m}{\sqrt{n_{1}n_{2}}}.$
\end{center}
To prove ${(\romannumeral 4)}$, using ${(\romannumeral 2)}$, we have, 
\begin{center}
    $\lambda_{n-1}\geq \eta_{2}= \frac{-m}{\sqrt{n_{1}n_{2}}}.$
\end{center}
This completes the proof.
\end{proof}
\begin{corollary}\label{t1cor1}
If $G$ is a regular bipartite graph then 
\begin{center}
    $\lambda_{2}\leq \frac{m}{n_{1}}$
\end{center}
\end{corollary}\label{t1cor2}
\begin{proof}
For a regular bipartite graph $G$, $n_{1}=n_{2}$. Substituting this in $(\romannumeral 3)$ of Theorem \ref{T1}, we get the result. 
\end{proof}
\begin{corollary}\label{t1cor3}
If $G$ is a regular bipartite graph then 
\begin{center}
    $\lambda_{n-1}\geq \frac{-m}{n_{1}}$
\end{center}
\end{corollary}
\proof
For a regular bipartite graph $G$, $n_{1}=n_{2}$. Substituting this in $(\romannumeral 4)$ of Theorem \ref{T1}, we get the result. 
\endproof
In the following theorems, we present simpler bounds, involving a single parameter $n$, when compared to the bounds obtained in Theorem \ref{T1}. To achieve this we require $m$ to be minimum. $m$ would be minimum only if $G$ is minimally connected, that is $m=n-1.$ \\We consider three types of bipartitions such as balanced, unbalanced and average bipartitions to arrive at our result. \\ \\For a balanced bipartition we have $(n_{1},n_{2})=\begin{cases}
(\frac{n-1}{2},\frac{n+1}{2}) &  \text{$n$ is odd}\medskip\\
(\frac{n}{2},\frac{n}{2}) & \text{$n$ is even}
\end{cases}$.\\ \\
Unbalaced bipartition is given by, $(n_{1},n_{2})=(1,n-1).$\\ \\
Average bipartition is given by\\
 $(n_{1},n_{2})=\begin{cases}(\frac{3n-4}{4},\frac{n+4}{4}) & \text{if $n$ is even}\medskip\\
(\frac{3n-3}{4},\frac{n+3}{4}) & \text{if $n=4k+3$ where $k=0,1,2,\dots.$}\medskip\\
(\frac{3n-5}{4},\frac{n+5}{4}) & \text{if $n=4k+1$ where $k=0,1,2,\dots.$}
\end{cases}$ \\
Let $\eta_{11}, \eta_{12}$ and $\eta_{13}$ represent the largest eigenvalues of the bipartite quotient matrices corresponding to the three different bipartitions.
Let $\eta_{1}=min \{\eta_{11}, \eta_{12}, \eta_{13}\}$.
Let $\eta_{21}, \eta_{22}$ and $\eta_{23}$ represent the smallest eigenvalues of the bipartite quotient matrices corresponding to the three different bipartitions. Let $\eta_{2}=max \{\eta_{21}, \eta_{22}, \eta_{23}\}$.
\begin{theorem}\label{t2}
Let $G$ be a connected bipartite graph of order $n$. Then  \begin{center}
    $\lambda_{2}\leq \begin{cases}
\frac{2(n-1)}{\sqrt{n^{2}-1}} &  \text{$n$ is odd}\medskip\\
\frac{2(n-1)}{n} & \text{$n$ is even.}
\end{cases}$\\
\end{center}
\end{theorem}
\proof
From Theorem \ref{T1} we have, 
\begin{equation}\label{EQ2}
 \lambda_{2}\leq \frac{m}{\sqrt{n_{1}n_{2}}}.   
\end{equation}\\
Consider the following cases.\\
\textbf{Case 1:}
Let $G$ be minimally connected with balanced bipartition. Substituting for $m, n_{1}$ and $n_{2}$ for a balanced bipartition in equation \eqref{EQ2}, we have
\begin{center}
$\eta_{11}= \begin{cases}
\frac{2(n-1)}{\sqrt{n^{2}-1}} &  \text{$n$ is odd}\medskip\\
\frac{2(n-1)}{n} & \text{$n$ is even}.
\end{cases}$
\end{center}
\textbf{Case 2:} Let $G$ be minimally connected with unbalanced bipartition. Substituting for $m, n_{1}$ and $n_{2}$ for an unbalanced bipartition in equation \eqref{EQ2}, we have
\begin{center}
$\eta_{12}=\frac{n-1}{\sqrt{n-1}}=\sqrt{n-1}.$
\end{center}
\textbf{Case 3:} Let $G$ be minimally connected with the average bipartition. Substituting for $m, n_{1}$ and $n_{2}$ for an average bipartition in equation \eqref{EQ2}, we have when $n$ is even,
\begin{center}
$\eta_{13}=\begin{cases}
\frac{4(n-1)}{\sqrt{3n^{2}+8n-16}} & \text{if $n$ is even}\medskip\\
\frac{4(n-1)}{\sqrt{3n^{2}+6n-9}} & \text{if $n=4k+3$ where $k=0,1,2,\dots.$}\medskip\\
\frac{4(n-1)}{\sqrt{3n^{2}+10n-25}} & \text{if $n=4k+1$ where $k=0,1,2,\dots.$}
\end{cases}$
\end{center}
When $n$ is even, comparing all the three cases, since the numerator contains $(n-1)$ as a common term we have
\begin{center}
$\sqrt{n-1}\leq \frac{\sqrt{3n^{2}+8n-16}}{4} \leq \frac{n}{2}$.
\end{center}
which implies that $\frac{2(n-1)}{n}\leq \frac{4(n-1)}{\sqrt{3n^{2}+8n-16}}\leq \sqrt{n-1}$. \\ \\
When $n$ is odd and $n=4k+3$ where $k=0,1,2,\dots,$ comparing all the three cases, since the numerator contains $(n-1)$ as a common term we have
\begin{center}
$\sqrt{n-1}\leq \frac{\sqrt{3n^{2}+6n-9}}{4} \leq \frac{\sqrt{n^{2}-1}}{2}$.
\end{center} which implies that $\frac{2(n-1)}{\sqrt{n^{2}-1}}\leq \frac{4(n-1)}{\sqrt{3n^{2}+6n-9}} \leq \sqrt{n-1}$.\\ \\
When $n$ is odd and $n=4k+1$ where $k=0,1,2,\dots,$ comparing all the three cases, since the numerator contains $(n-1)$ as a common term we have
\begin{center}
$\sqrt{n-1}\leq \frac{\sqrt{3n^{2}+10n-25}}{4} \leq \frac{\sqrt{n^{2}-1}}{2}$.
\end{center} which implies that  $\frac{2(n-1)}{\sqrt{n^{2}-1}}\leq \frac{4(n-1)}{\sqrt{3n^{2}+10n-25}} \leq \sqrt{n-1}$.\\ \\
Comparing all the three cases, we get $\eta_{11}\leq \eta_{13}\leq \eta_{12}.$
From this we have, $\eta_{1}=\eta_{11}$.
Applying this in equation \eqref{EQ2}, we get a tight upper bound in terms of $n$ as 
\begin{center}
    $\lambda_{2}\leq \begin{cases}
\frac{2(n-1)}{\sqrt{n^{2}-1}} &  \text{$n$ is odd}\medskip\\
\frac{2(n-1)}{n} & \text{$n$ is even}.
\end{cases}$\\
\end{center}
\endproof
\begin{theorem}\label{t9}
Let $G$ be a connected bipartite graph of order $n$. Then  \begin{center}
    $\lambda_{n-1}\geq \begin{cases}
\frac{-(2(n-1))}{\sqrt{n^{2}-1}} &  \text{$n$ is odd}\medskip\\
\frac{-(2(n-1))}{n} & \text{$n$ is even}.
\end{cases}$\\
\end{center}
\end{theorem}
\proof
In Theorem \ref{T1}, we have proved that 
\begin{equation}\label{EQ3}
 \lambda_{n-1}\geq\frac{-m}{\sqrt{n_{1}n_{2}}}.   
\end{equation}\\
Consider the following cases.\\
\textbf{Case 1:}  $G$ is minimally connected with balanced bipartition. Substituting for $m, n_{1}$ and $n_{2}$ for a balanced bipartition in equation \eqref{EQ3}, we have
\begin{center}
$\eta_{21}=\begin{cases}
\frac{-(4(n-1))}{\sqrt{3n^{2}+8n-16}} & \text{if $n$ is even}\medskip\\
\frac{-(4(n-1))}{\sqrt{3n^{2}+6n-9}} & \text{if $n=4k+3$ where $k=0,1,2,\dots.$}\medskip\\
\frac{-(4(n-1))}{\sqrt{3n^{2}+10n-25}} & \text{if $n=4k+1$ where $k=0,1,2,\dots.$}
\end{cases}$
\end{center}
\textbf{Case 2:} $G$ is minimally connected with unbalanced bipartition. Substituting for $m, n_{1}$ and $n_{2}$ for an unbalanced bipartition in equation \eqref{EQ3}, we have
\begin{center}
$\eta_{22}=\frac{-(n-1)}{\sqrt{n-1}}=-\sqrt{n-1}.$
\end{center}
\textbf{Case 3:}
$G$ is minimally connected with average bipartition. Substituting for $m, n_{1}$ and $n_{2}$ for an average bipartition in equation \eqref{EQ3}, we have
\begin{center}
$\eta_{23}= \begin{cases}
\frac{-(2(n-1))}{n} & \text{$n$ is even}\medskip\\
\frac{-(2(n-1))}{\sqrt{n^{2}-1}} &  \text{$n$ is odd}.
\end{cases}$
\end{center}
When $n$ is even, comparing all the three cases, since the numerator contains $(n-1)$ as a common term we have
\begin{center}
$\sqrt{n-1}\leq \frac{\sqrt{3n^{2}+8n-16}}{4} \leq \frac{n}{2}$. 
\end{center}
which implies that $\frac{-(2(n-1))}{n}\geq \frac{-(4(n-1))}{\sqrt{3n^{2}+8n-16}}\geq -\sqrt{n-1}.$ \\ \\
When $n=4k+3$ where $k=0,1,2,\dots,$ comparing all the three cases, since the numerator contains $(n-1)$ as a common term we have
\begin{center}
$\sqrt{n-1}\leq \frac{\sqrt{3n^{2}+6n-9}}{4} \leq \frac{\sqrt{n^{2}-1}}{2}$.
\end{center} which implies that $\frac{2(n-1)}{\sqrt{n^{2}-1}}\leq \frac{4(n-1)}{\sqrt{3n^{2}+6n-9}} \leq \sqrt{n-1}$. \\ \\
When $n=4k+1$ where $k=0,1,2,\dots,$ comparing all the three cases, since the numerator contains $(n-1)$ as a common term we have
\begin{center}
$\sqrt{n-1}\leq \frac{\sqrt{3n^{2}+10n-25}}{4} \leq \frac{\sqrt{n^{2}-1}}{2}$.
\end{center} which implies that  $\frac{2(n-1)}{\sqrt{n^{2}-1}}\leq \frac{4(n-1)}{\sqrt{3n^{2}+10n-25}} \leq \sqrt{n-1}$.\\ \\
Comparing all the three cases, we get $\eta_{21}\geq \eta_{23}\geq \eta_{22}.$
From this we have, $\eta_{2}=\eta_{21}$.
Applying this in equation (\ref{EQ3}), we get a tight lower bound in terms of $n$ as
\begin{center}
    $\lambda_{n-1}\geq \begin{cases}
\frac{-(2(n-1))}{\sqrt{n^{2}-1}} &  \text{$n$ is odd}\medskip\\
\frac{-(2(n-1))}{n} & \text{$n$ is even}.
\end{cases}$\\
\end{center}
\endproof
\subsection{Bounds for the Second largest Laplacian eigenvalue}\label{sec32}
In this section we derive the upper bound for the second largest Laplacian eigenvalue of a connected bipartite graph $G$.
\begin{theorem}\label{T5}
Consider a Bipartite graph $G$ with bipartition $V=(X,Y)$. Let $\mu_{1}\geq \mu_{2}\geq \dots\geq \mu_{n-1}\geq \mu_{n}$ be the eigenvalues of the Laplacian matrix $L$ of $G$, $L_{Q}$ be the bipartite quotient matrix of $L$ and $\theta_{1}$ and $\theta_{2}$ the eigenvalues of $L_{Q}.$ Then
\begin{center}
\begin{enumerate}
\begin{enumerate}
\begin{enumerate}
    \item $\mu_{1}\geq \theta_{1} \geq \mu_{2}$
    \item $\mu_{2} \leq \frac{mn}{n_{1}n_{2}}$
\end{enumerate}
\end{enumerate}
\end{enumerate}
\end{center}
\end{theorem}
\proof
The proof is similar to that of Theorem \ref{T1}.
\\
Let $L$ be the Laplacian matrix of $G$ represented in the following block matrix form with respect to the bipartition $V=(X,Y)$ \begin{center}
$L=\begin{bmatrix}
L_{11} & L_{12} \\
L_{21} & L_{22}
\end{bmatrix}.$
\end{center}
Let $L_{Q}$ be the Bipartite quotient matrix of the Laplacian matrix $L$ of $G$. Then
\begin{center}
$L_{Q}=
\begin{bmatrix}
\frac{m}{n_{1}} & \frac{-m}{n_{1}} \\
\frac{-m}{n_{2}} & \frac{m}{n_{2}}
\end{bmatrix}.$    
\end{center}
The characteristic equation of $L_{Q}$ is
$\mu^{2}-\frac{ mn}{n_{1}n_{2}}\mu=0$
Then the eigenvalues of $L_{Q}$ are $\theta_{1}=\frac{mn}{n_{1}n_{2}}$ and $\theta_{2}=0$.\\
To get the generalized interlacing consider the following.\\
The characteristic matrix of $G$ is given by 
$\tilde S= \begin{bmatrix}
J_{n_{1}\times 1}& 0_{n_{1}\times 1}\\
0_{n_{2}\times 1}& J_{n_{2}\times 1}
\end{bmatrix}_{n\times 2}$\\ 
Let $S=\tilde{S}K^{\frac{-1}{2}}$, where $K=diag(|X|,|Y|)$ i.e., $K=\begin{bmatrix}
n_{1} & 0 \\
0 & n_{2}
\end{bmatrix}.$ \\
Then $S= \begin{bmatrix}
P_{n_{1}\times 1}& 0_{n_{1}\times 1}\\
0_{n_{2}\times 1}& R_{n_{2}\times 1}
\end{bmatrix}_{n\times 2}.$\\
where $P_{n_{1}\times 1}$ and $R_{n_{2}\times 1}$ are column matrices with entries $(p_{i1})=\frac{1}{\sqrt{n_{1}}}$ where $i=1,2,\dots, n_{1}$ and  $(r_{j1})=\frac{1}{\sqrt{n_{2}}}$ where $j=1,2,\dots , n_{2}$.\\
From the proof of Lemma \ref{L1} \cite{WH2} we have,
\begin{equation}\label{EQ4}
    L_{Q}= S^{T}LS.
\end{equation}
Left and right multiplying equation \eqref{EQ4} by $S$ and $S^{T}$ we have \\$SL_{Q}S^{T}=L$. \\
Now consider $SL_{Q}S^{T}$ and denote it by $U$.\\
Let $U=(u_{ij})=\begin{cases}
\frac{m}{n_{1}^{2}} & \text{$v_{i},v_{j} \in X$}\medskip \\
\frac{m}{n_{2}^{2}} & \text{$v_{i},v_{j} \in Y$} \medskip \\
\frac{-m}{n_{1}\sqrt{n_{1}n_{2}}} & \text{$v_{i} \in X$ and $v_{j} \in Y$} \medskip \\
\frac{-m}{n_{2}\sqrt{n_{1}n_{2}}} & \text{$v_{i} \in Y$ and $v_{j} \in X$}
\end{cases}$ \\
Then the block matrix representation of $U$ is given by
\begin{center}
$U= \begin{bmatrix}
E & F \\
G & H 
\end{bmatrix}_{n\times n}.$
\end{center}
Let the eigenvalues of $U$ be $\beta_{1}\geq \beta_{2}\geq \dots \geq \beta_{n}.$ The eigenvalues of $U$ are the trace of  $U$ and $0$, that is $\beta_{1}=\frac{mn}{n_{1}n_{2}}$ and $\beta_{i}=0$ for $i=2,3,\dots,n$. \\
Then the interlacing becomes,\\
$\mu_{1}\geq \beta_{1}\geq \mu_{2}\geq \dots \geq \mu_{n-1}\geq \beta_{m}\geq \mu_{n}$ which implies that $\mu_{1}\geq \beta_{1}\geq \mu_{2}$.
By comparing the eigenvalues of $L_{Q}$ and $U$ we have, $\beta_{1}=\theta_{1}$.\\
Let us consider $\mu_{1}\geq \beta_{1}\geq \mu_{2}$. Since $\beta_{1}=\theta_{1}$, we have $\mu_{1}\geq \theta_{1} \geq \mu_{2}$. This proves $(\romannumeral 1).$
\\ To prove $(\romannumeral 2)$, using $(\romannumeral 1)$ we have,
\begin{center}
$\mu_{2} \leq \theta_{1}=\frac{mn}{n_{1}n_{2}}.$
\end{center}
\endproof
\begin{corollary}\label{lapcor}
If $G$ is a regular bipartite graph then $\mu_{2} \leq \frac{2m}{n_{1}}$.
\end{corollary}
\begin{proof}
For a regular bipartite graph $G$, $n_{1}=n_{2}$. Substituting this in Theorem \ref{T5}, we get the result. 
\end{proof}
\noindent \textbf{Note:}\\
For a complete bipartite graph $G$, the size of $G$ is equal to the product of the orders of the bipartitions of $G$, that is $m=n_{1}n_{2}$. Hence $\mu_{2} \leq n$.
\section{Vertex-split of a bipartite graph}\label{sec4}
In this section, we define vertex split of a bipartite graph. Next, we prove theorems related to the connectivity and expansion of vertex split of a bipartite graph.
\begin{defn}(Vertex-split of a bipartite graph)
Let $G=(X\cup Y, E)$ be a connected bipartite graph with $\delta\geq 4$ where $|X|=n_{1}\geq 4$, 
$|Y|=n_{2}\geq 3$ with $n_{1}> n_{2}.$ A graph $G'=
(X'\cup Y',E')$ is said to be a vertex-split of $G$ if
\begin{enumerate}
\item $|X'|=|X|$ and $|Y'|=|Y_{a}|+|Y_{b}|=2|Y|$ where $Y_{a}=\{y_{1a},y_{2a},\dots , y_{n_{1}a}\}$, $Y_{b}=\{y_{1b},y_{2b},\dots ,y_{n_{2}b}\}$.
\item Let $deg(y)=d_{y}=d_{a}+d_{b}$ $\forall y\in Y$ such that $d_{a}=deg(y_{a})$ $\forall y_{a}\in Y_{a}$  and $d_{b}=deg(y_{b})$ $\forall y_{b}\in Y_{b}$ with the condition that 
$|d_{a}-d_{b}|=\begin{cases}
1 & \text{if $d_{y}$ is odd}\\
0 & \text{if $d_{y}$ is even}.
\end{cases}$
    \item $N(Y_{a})\cap N(Y_{b})\neq \{\emptyset\}$
\end{enumerate}
\end{defn}
\textbf{Example:} Vertex-split of a complete graph $K_{8,4}$ is shown in figure \ref{fig:1} where $|X|=|X'|=8$, $|Y|=4$, $|Y'|=8$ \\

\begin{figure}[H]
   \centering
   \includegraphics[width=6cm]{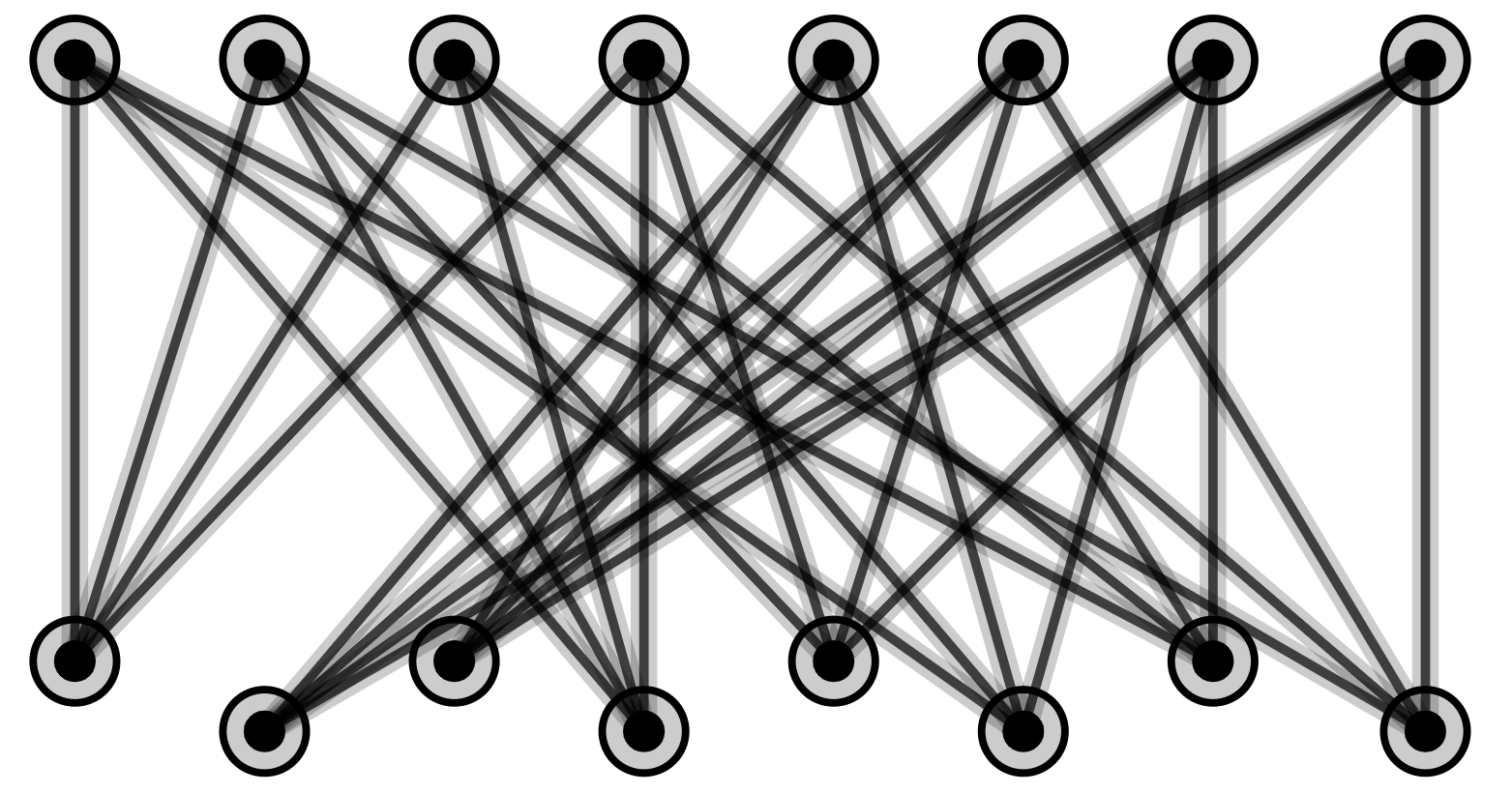}
    \caption{Vertex-split of $K_{8,4}$}
    \label{fig:1}
\end{figure}
\subsection{Connectivity of Vertex-split of a bipartite graph}\label{sec41}
Let $G'$ be a vertex-split of a bipartite graph $G$ with minimum degree $\delta'$ and edge connectivity $\kappa'(G')$. The adjacency eigenvalues of $G'$ are denoted as $\lambda_{1}'\geq \lambda_{2}'\geq \dots \geq \lambda_{n}'.$
\begin{theorem}\label{R1}
 Let $\delta' \geq k \geq 2$ be a constant and let $G(X\cup Y, E)$ be a $d-$regular bipartite graph and $G'$ the vertex-split of $G$. If 
 \begin{equation*}
\lambda_{2}'\geq \frac{2k-1}{\sqrt{2}}   
 \end{equation*}
 then $\kappa'(G')\geq k.$
\end{theorem}
\begin{proof}
Our proof is by contradiction.
We prove that if $G'$ is a connected bipartite graph of order $n'$ and minimum degree $\delta'$ such that $\kappa'(G')\leq 2k'-1$, then
\begin{equation}\label{abeq}
\lambda_{2}' \leq \frac{2k-1}{\sqrt{2}}.
\end{equation} 
 Let $n'=n_{1}'+n_{2}'$ where $n_{1}'=n_{1}$ and $n_{2}'=2n_{2}$ and $\delta'=\frac{\delta}{2}.$ Let $\delta'\geq k \geq 2.$ From this we have $\frac{\delta}{2}\geq k \geq 2.$ Consider $\frac{\delta}{2}\geq k.$ 
 By (\romannumeral 3) of Theorem \ref{T1} we have,
 \begin{equation}\label{abeq}
    \lambda_{2} \leq \frac{m}{\sqrt{n_{1}n_{2}}}.
\end{equation} 
 For a regular bipartite graph $n_{1}=n_{2}= \frac{n}{2}$ and $m=\frac{nd}{2}.$ Suppose $\delta \leq 2k-1.$
Substituting this in equation \ref{abeq} we get,
\begin{equation*}
\lambda_{2}'\leq \frac{d}{\sqrt{2}} = \frac{\delta}{\sqrt{2}}
\end{equation*}
\begin{equation*}
\lambda_{2}'\leq \frac{2k-1}{\sqrt{2}}  
\end{equation*}
This completes the proof.
\end{proof}

\textbf{Example for connectivity of vertex-split of a regular graph :}
Consider a complete bipartite graph $K_{5,5}.$ The vertex-split of $K_{5,5}$ is a graph with minimum degree $\delta'=2$ and the second largest eigenvalue $\lambda_{2}'=2.34.$ Since it is regular graph, from theorem \ref{T1} we have, for $k=2,$ $\frac{2k-1}{\sqrt{2}}=2.121$
\begin{figure}[H]
   \centering
   \includegraphics[width=8cm]{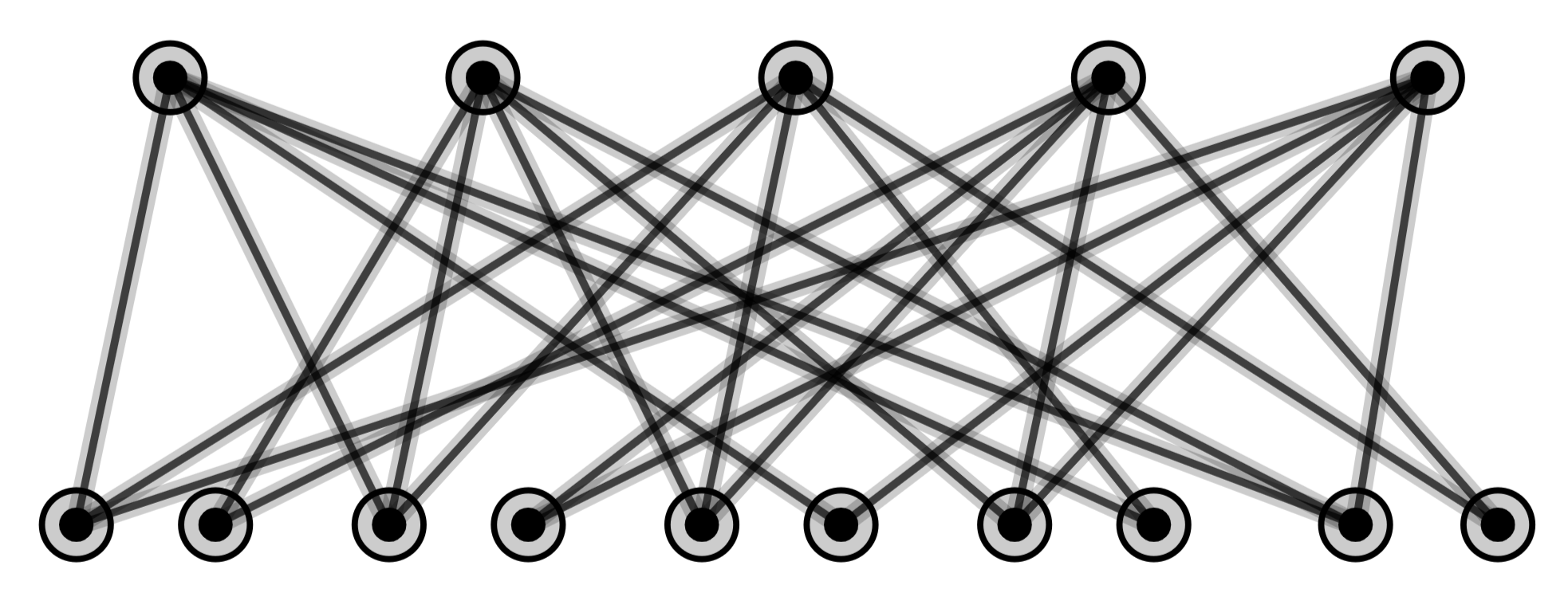}
    \caption{vertex-split of $K_{5,5}$}
    \label{fig:2}
\end{figure}

\begin{theorem}\label{r2}
Let $G$ be a biregular bipartite graph with $V=(X,Y)$  where $|X|=n_{1}$, $|Y|=n_{2}$ and $G'$ the vertex-split of $G$. If \begin{equation*}
\lambda_{2}' \geq \frac{n_{2}(2k-1)}{\sqrt{2n_{1}n_{2}}}
\end{equation*}
 then $\kappa'(G')\geq k.$ 
\end{theorem}
\begin{proof}
The proof is by contradiction. We prove that if $G'$ is a connected bipartite graph of order $n'$ and minimum degree $\delta'$ such that $\kappa'(G')\leq 2k-1$, then 
\begin{equation*}
\lambda_{2}'\geq \frac{n_{2}(2k-1)}{\sqrt{2n_{1}n_{2}}}.
\end{equation*}
 Let $n'=n_{1}'+n_{2}'$ where $n_{1}'=n_{1}$ and $n_{2}'=2n_{2}$ and $\delta'=\frac{\delta}{2}.$ Let $\delta'\geq k \geq 2.$ Therefore we have, $\frac{\delta}{2}\geq k \geq 2.$ Consider $\frac{\delta}{2}\geq k.$ Suppose $\delta \leq 2k-1.$\\
 For a biregular bipartite graph $n_{1}d_{1}=n_{2}d_{2}=m,$ where $d_{1}$ and $d_{2}$ are the degrees of the bipartition with $d_{1}>d_{2}.$ Similarly $n_{1}'d_{1}'=n_{2}'d_{1}'=m'=m.$
Substituting this in (\romannumeral 3) of Theorem \ref{T1} we have,
\begin{equation*}
 \lambda_{2}'\leq \frac{n_{2}d_{2}}{\sqrt{2n_{1}n_{2}}}= \frac{n_{2}\delta}{\sqrt{2n_{1}n_{2}}}
 \end{equation*}
 \begin{equation*}
 \lambda_{2}'\leq \frac{n_{2}(2k-1)}{\sqrt{2n_{1}n_{2}}}
\end{equation*}
This completes the proof.
\end{proof}
\subsection{Expansion of Vertex-split of a bipartite graph}\label{sec42}

\begin{theorem}\label{r4}
The vertex-split of a complete bipartite graph $K_{m,n}$ is a bipartite expander if
\begin{enumerate}
    \item $n$ is even
    \item $n\leq m \leq 2n$
\end{enumerate}
\end{theorem}
\begin{proof}
Let $G'=(X'\cup Y', E')$ be the vertex-split of the complete bipartite graph $G.$ Here $|X'|=|X|=m,$ $|Y'|=2|Y|=2n$. Assume that $S\subseteq X'$ with $|S|=\frac{n}{2}.$  
To prove this theorem, consider the following cases.\\
Case 1 : $m=n.$ \\ In this case, when $|S|=\frac{m}{2}$ the cardinality of the neighbourhood of $S$ is $|N(S)|=\frac{m+2}{2}$ which implies that $|N(S)|\geq \alpha |S|$ where $\alpha = (1+\frac{2}{n}).$\\
Case 2 : $m=2n.$\\ In this case, when $|S|=\frac{n}{2}$ the cardinality of the neighbourhood of $S$ is $|N(S)|=n$ which implies that $|N(S)|\geq \alpha |S|$ where $\alpha = 1.$\\
Case 3 : $n < m < 2n.$ \\In this case, when $|S|=\frac{n}{2}$ the cardinality of the neighbourhood of $S$ is $|N(S)|=n+(i-\frac{n}{2}),$ $i=2,3,\dots , \frac{n}{2}$ which implies that $|N(S)|\geq \alpha |S|$ where $\alpha = 1+\frac{(i-\frac{n}{2})}{n},$ where $i=2,3,\dots , \frac{n}{2}.$\\
\\ By definition (\ref{ver}), in all the cases we conclude that the vertex-split of $K_{m,n}$ is a bipartite expander.
\end{proof}
\noindent \textbf{Note:}\\
From the above theorem we say that, vertex-split $G'$ of $G$ is a $\alpha-$ vertex expander.
\begin{corollary}\label{r5}
    The vertex-split $G'$ of a complete bipartite graph $G=K_{m,n}$ where $n=\frac{m}{2}$ is a $(1-\gamma)-$spectral expander.
\end{corollary}
\begin{proof}
From the above theorem we know that $G'$ is a $\alpha-$ vertex expander where $\alpha=1+\epsilon$ where $\epsilon=\frac{1}{n}.$
From lemma \ref{T2}, we can conclude that $G'$ is $(1-\gamma)-$spectral expander where $\gamma=\frac{1}{n^{2}d'},$  where $d'$ is the degree of $G'.$
\end{proof}
\section{Construction of Error correcting code}\label{sec5}
In this section, we prove that the vertex split of a complete bipartite graph is a bipartite expander and construct en error correcting code with good expansion factor.\\ 
Consider a biregular graph $G$ with bipartitions $(X,Y)$ and the corresponding degrees $d_{1}$ and $d_{2}.$ Let $d_{1}=\frac{|X|}{2}.$ 
In this section, we prove that the vertex-split of a biregular bipartite graph is a bipartite expander and then we show that the vertex-split of a biregular bipartite graph forms an expander code.\\
Let $C(G)$ be the error correcting code of block length $n$ obtained from a bipartite graph $G$ and let $\Delta(C(G))$ be the distance of the code $C.$
\begin{lemma}\label{lemt}
Let $G=(X\cup Y, E)$ be a biregular bipartite graph with $|X|=n_{1}$ and $|Y|=n_{2}.$ Let $d_{1}, d_{2}$ be the degrees of the corresponding bipartitions $X,Y$ where $d_{1}=\lfloor \frac{|X|}{2}\rfloor .$ Then vertex-split $G'$ of $G$ is a bipartite expander.
\end{lemma}
\begin{proof}
Let $G'=(X'\cup Y',E')$ be the vertex-split of a biregular bipartite graph $G$ with $|X'|=n_{1}$ and $|Y'|=2n_{2}.$ The degrees of the bipartitions $X'$, $Y'$ are denoted by $d_{1}'$ and $d_{2}'$ respectively. Let $S$ be a subset of $X'$ with $|S|=\frac{n_{1}}{d_{1}'}.$ Since $d_{1}'=\frac{|X'|}{2}$, we have $|S|=2.$ For any subset $S$ of $X'$ the neighbourhood of $S$ is $|N(S)|\geq \frac{n_{1}}{2}+1,$ which implies that $\frac{|N(S)|}{|S|}\geq \alpha$ where $\alpha=(\frac{n_{1}+2}{4})$ is the vertex expansion constant.\\ That is $|N(S)|\geq \alpha |S|$.
This proves the result.
\end{proof}
\begin{theorem}\label{r6}
 The vertex-split of a biregular bipartite graph $G=(X\cup Y, E)$ where $|X|=n_{1}$, $|Y|=n_{2}$ with degree $d_{1}=\frac{|X|}{2}$ is an expander code.
\end{theorem}
\begin{proof}
We proved that the vertex-split of a biregular bipartite graph $G$ is a bipartite expander in Lemma \ref{lemt}. For constructing an efficient error correcting code, let us fix $|Y|$ as $|Y|=n_{2}=max \{x : x|\frac{n_{1}^2}{2}$ and $\frac{n_{1}+2}{4}<x< \frac{n_{1}}{2}\}.$ Usually $\alpha$ is close to $\frac{D}{2}$ where $D$ is the degree of the larger side partition. Here $D=d_{1}=\frac{n_{1}}{2}.$ Then $\frac{D}{2}=\frac{n_{1}}{4}.$ Clearly $\frac{D}{2}<\alpha<D,$ where $\alpha=(\frac{n_{1}+2}{4}).$ Putting $\alpha=D(1-\epsilon),$ we get $\epsilon=\frac{n_{1}-2}{2n_{1}}<\frac{1}{2}.$
\end{proof}
\begin{corollary}\label{cor8}
Let the vertex-split $G'$ of $G$ be a $(n_{1},n_{2},d_{1},\gamma, D(1-\epsilon))$ expander. Then the distance of the ECC corresponding to graph $G'$ is $\Delta(C(G'))\geq \frac{n_{1}(n_{1}+2)}{2d_{1}^{2}}.$
\end{corollary}
\begin{proof}
We proved in Lemma \ref{lemt} that the vertex-split $G'$ of $G$ is a 
$(n_{1},n_{2},d_{1},\gamma, D(1-\epsilon))$ expander, where $\gamma=\frac{1}{d_{1}}, \alpha=\frac{n_{1}+2}{4}$ and $\epsilon=\frac{n_{1}-2}{2n_{1}}.$ 
Substituting all these in Lemma \ref{dist}, we get the result.
\end{proof}
\section{Conclusion}
In the first part of this paper, we defined bipartite quotient matrix of matrices associated with bipartite graphs and proved that the two eigenvalues of the bipartite quotient matrix interlace at the two ends of the spectrum of the adjacency matrix of a bipartite graph. From this interlacing, we obtained the upper bound for the second largest eigenvalue and the lower bound for the second smallest eigenvalue of a bipartite graph. To achieve a more sharper bound we considered the minimally connected graph for three different bipartitions and obtained two new bounds which depend only on the order of the graph. Also, we proved the eigenvalue interlacing for the Laplacian matrix and obtained an upper bound for the second largest Laplacian eigenvalue.\\
In the second part of this paper, we defined vertex-split of a bipartite graph and and proved its connectivity with respect to the second largest eigenvalue. Finally, we proved that the vertex-split of a complete bipartite graph $K_{m,n}$ is a bipartite expander and constructed an efficient ECC with expansion factor $\alpha> \frac{D}{2}$ corresponding to the vertex-split of $K_{m,n}.$


\begin{thebibliography}{00}
\bibitem{AL}{N. Alon. Eigenvalues and expanders. Combinatorica, Vol. 6, pages 83–96,
1986.}
\bibitem{Cap}{Capalbo, Michael, Omer Reingold, Salil Vadhan, and Avi Wigderson. 2002. Randomness
conductors and constant-degree lossless expanders. In Proceedings of the 34th Annual ACM
Symposium on Theory of Computing, Montreal, Quebec, Canada, May 19-21, 2002 (STOC `02),
659-668. New York: ACM.}
\bibitem{GG}{O. Gabber and Z. Galil. Explicit constructions of linear-sized superconcentrators. Journal of
Computer and System Sciences, 22(3):407–420, June 1981.}
\bibitem{VG}{Guruswami, Venkatesan, Christopher Umans, and Salil Vadhan. 2009. “Unbalanced Expanders
and Randomness Extractors from Parvaresh-Vardy Codes.” Journal of the ACM 56, 4: 1–34.}
\bibitem{HLW}{ S. Hoory, N. Linial, and A. Wigderson. Expander Graphs and their Applications. Bull. AMS, Vol. 43 (4), pages 439–561, 2006.}
\bibitem{LPS}{A. Lubotzky, R. Phillips, and P. Sarnak. Ramanujan graphs. Combinatorica, 8(3):261–277, 1988.}
\bibitem{Mar1}{G. A. Margulis. Explicit constructions of expanders. Problemy Peredaˇci Informacii, 9(4):71–80,
1973.}
\bibitem{BH}
Andries E. Brouwer, Willem H. Haemers: Spectra of graphs. Springer Publications (2011).
\bibitem{CH}
Chang An. \newblock Bounds on the second largest eigenvalue of a tree with perfect matchings. \newblock{\em Linear Algebra and its Applications} 247-255 (1998).
\bibitem{CHL}
Chia-an Liu, Chih-wen Weng. \newblock Spectral radius of bipartite graphs. 23 Feb 2014.
\bibitem{CV}
D. M. Cvetcovic, M. Doob, and H. Sachs.
\newblock Spectra of Graphs. Academic Press, Inc., New York
(1980).
\bibitem{WH1} {W. H. Haemers: Eigenualue Techniques in Design and Graph Theoy. Math.Centre Tract 121, Mathematical Centre, Amsterdam, (1980).}
\bibitem{WH2}
W.H.Haemers. 
\newblock Interlacing eigenvalues and graphs. \newblock{\em Linear Algebra Appl}. 226–228 (1995) 593–616.
\bibitem{HMF} 
Huiqing Liu, Mei Lu, Feng Tian.
\newblock Edge-connectivity and (signless) Laplacian eigenvalue of graphs. \newblock{\em Linear Algebra and its Applications}. 439 (2013) 3777–3784.
\bibitem{MHB}
Mingqing Zhai a, Huiqiu Linb, BingWanga. 
\newblock Sharp upper bounds on the second largest eigenvalues of connected graphs. \newblock{\em Linear Algebra and its Applications} 437 (2012) 236–241.
\bibitem{PS}
\newblock Patrik Solb.
\newblock The Second Eigenvalue of Regular Graphs of Given Girth. \newblock{\em Journal of combinatorial theory, Series B} 56, 283 (1992) 239-249.
\bibitem{PW}
D.L. Powers. 
\newblock Graph partitioning by eigenvectors. 
\newblock{\em Linear Algebra Appl.} 101 (1988) 121–133.
\bibitem{RR}
Ranjit Mehatari and M.Rajesh Kannan. \newblock Eigenvalue bounds for some classes of matrices associated with graphs. \newblock {\em Czechoslovak Mathematical Journal}, 71 (146) (2021), 231–251.
\bibitem{ss}{Sipser M and Spielman D.A. (1996). "Expander codes". IEEE Transactions on Information Theory. 42 (6): 1710–1722.}
\bibitem{daniel}{Spielman, Daniel A. "Linear-time encodable and decodable error-correcting codes." In Proceedings of the twenty-seventh annual ACM symposium on Theory of computing, pp. 388-397. 1995.}
\bibitem{zemor}{Zémor, Gillés. "On expander codes." IEEE Transactions on Information Theory 47, no. 2 (2001): 835-837.}
\bibitem{tan}{Tanner, R. "A recursive approach to low complexity codes." IEEE Transactions on information theory 27, no. 5 (1981): 533-547.}
\bibitem{vertosp}{https://resources.mpi-inf.mpg.de/departments/d1/teaching/ss13/gitcs/lecture7.pdf}
\bibitem{guru}{https://www.cs.cmu.edu/~venkatg/teaching/codingtheory/notes/notes8.pdf}
\end{thebibliography}
\end{document}